%% file: twistWRlattices.tex
\documentclass[12pt,a4paper,leqno]{amsart}
\usepackage[tmargin=2.5cm,bmargin=2.5cm,lmargin=2.5cm,rmargin=2.5cm]{geometry}
\usepackage{amsmath}
\usepackage{xcolor}
\usepackage{amsthm}
\usepackage{amsfonts}
\everymath{\displaystyle}
\usepackage{cite,graphicx}
\usepackage[unicode]{hyperref}

\usepackage{enumerate}
\usepackage{tcolorbox}
\newtheorem{theorem}{Theorem}
\newtheorem{definition}{Definition}
\newtheorem{corollary}{Corollary}
\newtheorem{lemma}{Lemma}
\newtheorem{proposition}{Proposition}
\newtheorem{example}{Example}
\newtheorem{remark}{Remark}
\newtheorem{algorithm}{Algorithm}
\usepackage{listings}
\usepackage{color}
\definecolor{dkgreen}{rgb}{0,0.6,0}
\definecolor{gray}{rgb}{0.5,0.5,0.5}
\definecolor{mauve}{rgb}{0.58,0,0.82}
\DeclareMathOperator{\vol}{vol}

\title{\textbf{Well-Rounded Twists of Ideal Lattices from Imaginary Quadratic Fields}}
\author{Nam H. Le, Dat T. Tran, Ha T. N. Tran}

\begin{document}
	\maketitle
	
\begin{abstract} 
In this paper, we investigate the properties of well-rounded twists of a given ideal lattice  of an imaginary quadratic field $K$. We show that every ideal lattice $I$ of $K$ has at least one well-rounded twist lattice. Moreover, we provide an explicit algorithm to compute all well-rounded twists of  $I$.
\end{abstract}

\section{Introduction}
\input{introduction}

\section{Background}\label{sec_backgd}
\input{background}


\section{Main results}\label{sec_main}
\input{mainresult}


\section{Algorithms and a numerical example} \label{sec_alg}
\input{alg_example}


\section{Analysis of Algorithm 1 and Algorithm 2}\label{sec_analysis}
\input{analysis}

\section*{Acknowledgement}
The third author acknowledges the support of the Natural Sciences and Engineering Research Council of Canada (NSERC) (funding RGPIN-2019-04209 and DGECR-
2019-00428).

\bibliographystyle{abbrv}
\bibliography{myrefs}

\end{document}

%% file: introduction.tex
A lattice of full rank in a Euclidean space is called \textit{well-rounded} if its set of minimal vectors spans the whole space. Well-rounded lattices are important in discrete optimization, in particular in the study of sphere packing, sphere covering, and kissing number problems \cite{martinet2013perfect}, as well as in coding theory \cite{banihashemi1998, WR1, WR2, damir2018analysis}. 

A \textit{well-rounded twist} of a lattice is defined in  \cite{damir2019well}. A method for computing all well-rounded twists of a given ideal lattice $I$ of a real quadratic field $K$ is also presented in this paper. It requires us to compute all principal ideals $\langle x \rangle \subset I$ such that $N(x)^2 \le N(I)^2 \Delta_K/3$ and its generator $x$ where $\Delta_K$ is the discriminant of $K$ (see Section 3 in \cite{damir2019well}). It is known that finding all ideals of norm bounded and finding a generator of a principal ideal are hard problems, especially when $\Delta_K$ is large (see \cite{lenstra1992algorithms}).  This method is therefore infeasible and hence one cannot always compute all well-rounded twists of a given ideal lattice $I$. 
 In contrast, we can show that this task is  feasible for an arbitrary imaginary quadratic field $K$. Indeed, in this paper we prove that every ideal lattice $I$ of $K$ has at least one well-rounded twist lattice (see Proposition \ref{theo6}). This result can be considered as a particular case of the one in  \cite{solan2019stable} which proves that every lattice (in any dimension) has at least a well-rounded twist. However, in our proof, we make use of an independent idea and argument from the ones in \cite{solan2019stable}. Indeed Proposition \ref{theo6} is implied from the proofs of Lemma \ref{bode3} and Theorem \ref{theo5}. We remark that a similar result for real quadratic fields has not been proved in \cite{damir2019well}.  Moreover, we provide algorithms to compute all well-rounded twists of  $I$ (see Section \ref{sec_alg}). In particular, we give an upper bound for the number of such well-rounded twists (see Corollary \ref{numbertwists}). The main idea is as below.

Let $K$ be an imaginary quadratic field and let $I$ be an integral ideal of $K$ with a $\mathbb{Z}$-basis $B=\{u,v\}$. 
We define the function 
\begin{align*}
  F(B)=F(u,v)=\dfrac{1}{4}\left[\left(\Im(u^2)+\Im(v^2)\right)^2-\Im(uv)^2\right] 
\end{align*}
here we denote by $\Im (z)$ the imaginary part of $z \in \mathbb{C}$ (see \ref{(3)} for an explicit formula). Note that a well-rounded twist of  $I$ is determined by a good basis of $I$ (see Definition \ref{def_goodbasis}). 
By Proposition \ref{theo2}, the basis $B$ is good if and only if $F(B)\le 0$. This inequation only has finitely many solutions as a result of Proposition \ref{theo3}. In addition, it provides us a necessary condition  for finding all good bases $B$, that is the imaginary part of $u^2$ and of $v^2$, denoted by $\Im (u^2)$ and $\Im (v^2)$, in absolute value are at most $\vol(I)$ (see ii) in Proposition \ref{theo3}). Thus one first lists all elements $x \in I$ such  that  
$\left(\Im (x^2)\right)^2\le \vol^2(I)$. Theorem \ref{theo4} says that there are only finitely many possibilities for $x$. After that, for each $x$ found, we solve $F (x, y) \le 0$ for all possible $y$ such that $\{x, y\}$ is a good basis of $I$. Finally, Theorem \ref{theo5} shows that given such an $x$, there are at most two good bases of the form $\{x, y \}$, up to similarity. The proof of this theorem also gives us  explicit formulae to compute those bases.  Hence, we can construct all good bases of the ideal $I$  demonstrated in Section \ref{sec_alg}.

Employing the algorithms presented in Section \ref{sec_alg}, we can first find all well-rounded twists of an ideal lattice $I$ of $K$ and after that check which lattices are similar using Remark \ref{similarity}. A natural question arisen from our work is how to compute only similar classes instead of all well-rounded twists of $I$. In other words, assume that we have computed  a list $\mathcal{L}$ of some well-rounded twists lattice of $I$, we would like to find a method to eliminate well-rounded twists $J$ that are similar to the ones in $\mathcal{L}$ before explicitly computing a basis for $J$. Another question is how to compute all well-rounded twists of ideal lattices in higher degree number fields. These open questions requires us a  further research in the future.

The structure of this paper is as follows. In Section \ref{sec_backgd}, we recall some basic definitions and properties of well-rounded lattices in $\mathbb{R}^2$ as well as of good bases and twists of a lattice. Our main results (Theorem \ref{theo4}, Theorem \ref{theo5} and Theorem \ref{theo6}) are presented in Section \ref{sec_main}. Based on the results of this section, we construct algorithms to compute all well-rounded twists of an ideal $I$ of $K$ and demonstrate them by an example in Section \ref{sec_alg}. We prove the correctness and analyze the complexity of these algorithms in Section \ref{sec_analysis}.
\newpage

%% file: background.tex
In this section we recall some basic definitions and properties of well-rounded lattices (twists) in $\mathbb{R}^2$, and then of well-rounded ideal lattices arisen from imaginary quadratic fields.

\begin{definition}
	Two planar lattices $\Lambda_1, \Lambda_2 \subset \mathbb{R}^2$ are called similar, denoted $\Lambda_1 \sim \Lambda_2$, if there exists a positive real number $\alpha$ and a $2\times2$ real orthogonal matrix $U$ such that $\Lambda_2 = \alpha U\Lambda_1$.
\end{definition}
See \cite{fukshansky2013well} for more details.  Moreover, if $B$ is a basis of $\Lambda_1$ then $\alpha UB$ is a basis of $\Lambda_2$ and if $B'$ is a basis of $\Lambda_2$ then there exists a basis $B$ of $\Lambda_1$ such that $B'=\alpha UB$ (we call $B$ and $B'$ are two similar bases). Thus, one has the following result.


 \begin{proposition}\label{propo}
	Suppose $\Lambda_1, \Lambda_2$ are two lattices of $\mathbb{R}^2$. Then $\Lambda_1\sim \Lambda_2$ if and only if there exist bases $B_1=\left\{x_1,y_2\right\}$ and $B_2=\{x_2,y_2\}$ of $\Lambda_1$ and $\Lambda_2$ respectively such that $\left|\cos(x_1,y_1)\right|=\left|\cos(x_2,y_2)\right|$ and $\dfrac{\|x_1\|}{\|y_1\|}=\dfrac{\|x_2\|}{\|y_2\|}$, where $(x_i, y_i)$ is the angle between two vectors $x_i$ and $y_i, i=1, 2$.
\end{proposition}


In this section, we will denote by $\Lambda$ a lattice in $\mathbb{R}^2$ and by $B=\{x,y\}$ a basis of $\Lambda$ with $x=(a,c), y=(b,d) \in \mathbb{R}^2.$

 \begin{definition}
	For each real number $\alpha>0$, we define the matrix \begin{align*}
	T_\alpha=\begin{bmatrix}
	\alpha&0\\0&\dfrac{1}{\alpha}
	\end{bmatrix}.
	\end{align*}
	The lattice $T_\alpha\Lambda$ is called \textbf{the twisting lattice} or \textbf{the twist} of $\Lambda$ with respect to $\alpha$.
\end{definition}


\begin{proposition}\label{theo1}
	Let $B=\left\{x=(a,c) , y=(b,d) \right\}$ be a basis of a lattice $\Lambda$. Then for all $\alpha>0$, there exist $x'=(e,f)\in \mathbb{R}^2$ and $f>0$ such that the lattice generated by $T_\alpha B$ is similar to the lattice generated by $B'=\left\{(1,0) ,x'\right\}$.
\end{proposition}

\begin{proof}
	By Proposition \ref{propo}, we prove that there exists $x'\in \mathbb{R}^2$ such that $\dfrac{\|T_\alpha x\|}{\|T_\alpha y\|}=\dfrac{1}{\|x'\|}$ and $\left|\cos\left(T_\alpha x, T_\alpha y\right)\right|$ is equal to the absolute value of the cosine of  the angle between $(1,0)$ and $x'$. We consider 
	\begin{align*}
	z=\dfrac{1}{\alpha^4a^2+c^2}\left(ab\alpha^4+cd,\alpha^2(ad-bc)\right),
	\end{align*}
	
	then $
	\|z\|=\dfrac{1}{\alpha^4a^2+c^2}\sqrt{(ab\alpha^4+cd)^2+\alpha^4(ad-bc)^2}
	=\sqrt{\dfrac{\alpha^4 b^2+d^2}{\alpha^4a^2+c^2}}
	=\dfrac{\|T_\alpha y\|}{\|T_\alpha x\|}.
	$
	
	It also implies that
$	\|-z\|=\dfrac{\|T_{\alpha}y\|}{\|T_{\alpha}x\|}$. There are two cases: 
	
	\begin{itemize}
		\item \textbf{Case 1:} If $ad-bc>0$, we choose $x'=z$.
		\item \textbf{Case 2:} If $ad-bc<0$, we choose $x'=-z$.
	\end{itemize}

For both cases,  it is clear that $\left|\cos\left(T_\alpha x, T_\alpha y\right)\right|$ is equal to the absolute value of the cosine of  the angle between $(1,0)$ and $x'$. 
\end{proof}

\begin{remark}
This proposition is shown in particular case of the following statement: every planar lattice is similar to a lattice with a basis $B'$ as described in Proposition \ref{theo1} and hence planar lattices are identified with $SO_2(\mathbb{R})\backslash SL_2(\mathbb{R})/SL_2(\mathbb{Z})$ (see \cite{solan2019stable}). 

\end{remark}
The idea of Proposition \ref{theo1} is similar to \cite{damir2019well} (see the function in [15] in this paper).  For each given basis $B$, there are two vectors satisfying Proposition \ref{theo1}. However, we can add the condition that $\|x'\|\ge 1$ to have the uniqueness of $x'$. We denote $x'$ by $\tau(\alpha,B)$ to emphasize that it depends on $\alpha$ and $B$.

\begin{definition}
	Let $\Lambda$ be a lattice in $\mathbb{R}^2$.
	\begin{enumerate}[i)]
		
		\item The \textbf{set of the  minimal vectors} of $\Lambda$ is
		\begin{align*}
		S(\Lambda)=\left\{x\in \Lambda:\|x\|=\lambda_1(\Lambda)\right\},
		\end{align*}
		where   $\lambda_1(\Lambda)=\min_{0\ne x\in \Lambda}\left\|x\right\|$.
		
		\item The lattice $\Lambda$ is called \textbf{well-rounded} if $\text{span}_\mathbb{R}(S(\Lambda))=\mathbb{R}^2$.
		\item If $\Lambda$ is well-rounded and $\{x_1,x_2\}$ is its basis such that $x_1,x_2 \in S(\Lambda)$, then $\{x_1,x_2\}$ is called a \textbf{minimal basis} of $A$.
		\item A basis $\left\{x_1,x_2\right\}$ is \textbf{twistable} if there exists a matrix $T_\alpha$ such that $\|T_\alpha x\|=\|T_\alpha y\|$.
	\end{enumerate}
\end{definition}

\begin{remark}
Note that if $S(\Lambda)$ contains two independent vectors then these vectors form a minimal basis for $\Lambda$. This fact is not true for lattices of higher dimension (see \cite{nguyen2009hermite} for more details).
\end{remark}

The following lemma is a result of Proposition \ref{propo} and the definition of well-rounded lattices.

\begin{lemma}\label{similarwr}
Two well-rounded lattices are similar if and only if there exist their minimal bases $B_1=\left\{x_1,y_2\right\}$ and $B_2=\{x_2,y_2\}$  such that $\left|\cos(x_1,y_1)\right|=\left|\cos(x_2,y_2)\right|$.
\end{lemma}


\begin{proposition}\label{prop2}
	Let $\beta=\dfrac{d^2-c^2}{a^2-b^2}$. Then  $B$ is twistable if and only if $\beta>0$. If this is the case, then $\alpha$ is unique and $\alpha=\beta^{1/4}$.
\end{proposition}

\begin{proof}
	See Proposition 1 of \cite{damir2019well}.
\end{proof}


To emphasize that $\alpha$ and $\beta$ are functions depending on 
the basis $B$ of $\Lambda$, we will write $\alpha_\Lambda(B)$ and $\beta_\Lambda(B)$ for $\alpha,\beta$. Note that $\alpha_\Lambda(B)=\left(\beta_\Lambda(B)\right)^{1/4}$.


\begin{proposition}\label{prop3}
	If $B$ is a twistable basis with the twisting matrix $T_\alpha$, then \begin{align*}
	\cos \theta_{T\alpha B}=\dfrac{ac+bd}{ad+bc}.
	\end{align*}
	
\end{proposition}

\begin{proof}
	See Proposition 2 in \cite{damir2019well}.
\end{proof}

\begin{proposition}\label{prop4}
	If $\left|\dfrac{ac+bd}{ad+bc}\right|\le \dfrac{1}{2}$ then $B$ is twistable.
\end{proposition}

\begin{proof}
	See Proposition 3 of \cite{damir2019well}.
\end{proof}


In this case, since $|\cos\theta_{T_\alpha B}|\le \dfrac{1}{2}$ where $\beta=\dfrac{d^2-c^2}{a^2-b^2}$ and $\alpha =\beta^{1/4}$,  the lattice $T_\alpha \Lambda$ is well-rounded. Moreover, $\left\{T_\alpha x, T_\alpha y\right\}$ is a minimal basis of $T_\alpha \Lambda$.

\begin{definition}\label{def_goodbasis}
	We call a basis $B$ of $\Lambda$ \textbf{good for twisting} or a \textbf{good basis}  if \begin{align}\label{(1)}
	\left|\dfrac{ac+bd}{ad+bc}\right|\le \dfrac{1}{2}.
	\end{align}
\end{definition}

 This definition is equivalent to the following statement: there exists $\alpha>0$ such that $T_\alpha \Lambda$ is well-rounded with a minimal basis $T_\alpha B$.\\
  By transforming inequation \eqref{(1)}, a basis $B$ is a good basis if 
 \begin{align}\label{(2)}
a^2c^2+abcd+b^2d^2-\dfrac{(ad-bc)^2}{4}\le 0 \textrm{ and }ad+bc\ne 0.
\end{align}
From the first inequality of \eqref{(2)}, one may define the polynomial 
\begin{align}
\nonumber
 F(B)=&(ac)^2+abcd+(bd)^2-\dfrac{(ad-bc)^2}{4}\\
\label{(3)}
=&\left(ac+bd+\dfrac{ad+bc}{2}\right)\left(ac+bd-\dfrac{ad+bc}{2}\right).
\end{align}
Frow now, we only need to consider the basis $B=\{(a,c);(b,d)\}$ where $ad+bc\ne 0$.


We have the following result that is similar to Theorem 1 in \cite{damir2019well}.

\begin{proposition}\label{theo2}
	Let $\Lambda$ be a lattice in $\mathbb{R}^2$. 
	\begin{enumerate}[(i)]
		\item A basis $B$ is good for twisting if and only if $F(B)\le 0$.
		\item  If B is good for twisting, then $		\min \{ (ac)^2, (bd)^2 \} \le\dfrac{\vol^2(\Lambda)}{4}.$
	\end{enumerate}
\end{proposition}


\begin{proof}
	\begin{enumerate}[(i)] 
		\item This fact can be easily implied from the definition of $F$.
		\item Recall that $\text{vol}^2(\Lambda)=(ad-bc)^2$. Since $B$ is good for twisting, it follows that \begin{align*}
		0\ge (ac)^2+abcd+(bd)^2-\dfrac{\vol^2(\Lambda)}{4}&=\dfrac{(ac)^2+(bd)^2}{2}+\dfrac{(ac+bd)^2}{2}-\dfrac{\vol^2(\Lambda)}{4}\\&\ge  \dfrac{(ac)^2+(bd)^2}2-\dfrac{\text{vol}^2(\Lambda)}{4}.
		\end{align*}
		Hence $\min\left\{(ac)^2,(bd)^2\right\}\le \dfrac{\vol^2(\Lambda)}{4}$.
	\end{enumerate}
\end{proof}


For $D$ positive and squarefree, we put $K=\mathbb{Q}(\sqrt{-D})$ and
\begin{align*}
\delta=\left\{\begin{matrix}
&\sqrt{-D},&\text{ if } -D\not\equiv 1\mod 4\\&\dfrac{1+\sqrt{-D}}{2},&\text{ if }-D\equiv 1\mod 4
\end{matrix}\right..
\end{align*}
The ring of integers of $K$ is $\mathcal{O}_K=\mathbb{Z}\left[\delta\right]$.
The embeddings $\sigma_1,\sigma_2: K\longrightarrow \mathbb{C}$ are given by 
\begin{align*}
    \sigma_1\left(x+y\delta\right)=x+y\delta,\hspace*{0.3cm} \sigma_2(x+y\delta)=\left\{\begin{matrix}x-y\delta\textrm{ if }D\not\equiv 1 \pmod 4\\x+y(1-\delta)\textrm{ if } D\equiv1\pmod 4.
    \end{matrix}\right.    
\end{align*}
We have $\sigma_2=\overline{\sigma_1}$. Hence we denote by $\sigma_K$ the embedding from $K$ into $\mathbb{R}^2$ defined by $\sigma_K=\left(\Re \sigma_2, \Im \sigma_2\right),$ 
where $\Re$ and $\Im$ stand for the real and imaginary parts, respectively.\\ 
Now let $I\subset \mathcal{O}_K$ be an ideal and  $\left\{t,y+g\delta \right\}$ its $\mathbb{Z}$-basis where $0\le y<t$, $g|y,t$ and $0< g\le t$. This basis is called the \textbf{canonical basis} of $I$. \\
Suppose $u,v \in I$ form a $\mathbb{Z}$-basis of $I$, then we can represent the lattice $\Lambda_K(I)=\sigma_K(I)$ as $$\Lambda_K(I)=\left(\begin{matrix}
&\Re(u)&\Re(v)\\&-\Im(u)&-\Im(v)
\end{matrix}\right)\mathbb{Z}^2=\left(\begin{matrix}
\dfrac{u+\overline{u}}{2}&\dfrac{v+\overline{v}}{2}\\\dfrac{\overline{u}-u}{2i}&\dfrac{\overline{v}-v}{2i}
\end{matrix}\right)\mathbb{Z}^2.$$
where $\Re(z)=\frac{z+\overline{z}}{2},\Im(z)=\frac{z-\overline{z}}{2i}$.


\begin{proposition}\label{prop5}
	Let $B=\{u,v\}$ be a twistable basis of $I$. Then $\cos \theta_{T_\alpha B}=\dfrac{\Im(u^2)+\Im(v^2)}{2\Im(uv)}.$\end{proposition}

\begin{proof}
	Proposition \ref{prop3} provides that \begin{align*}
	\cos T_\alpha B=\dfrac{\dfrac{\overline{u}^2-u^2}{4i}+\dfrac{\overline{v}^2-v^2}{4i}}{\dfrac{\overline{uv}-uv}{2i}}=\dfrac{\Im(u^2)+\Im(v^2)}{2\Im(uv)}.
	\end{align*}
\end{proof}

%% file: mainresult.tex
From now on, we follow the notations used in Section \ref{sec_backgd}.
The second inequality of \eqref{(2)} becomes
\begin{align}\label{condition}
    \Re(u)\Im(v)+\Re(v)\Im(u)\ne 0
\end{align}

Applying Proposition \ref{theo2}, we obtain a similar result as the one in Theorem 2 in \cite{damir2019well} as below.


\begin{proposition}
	\label{theo3}
	Let $I$ be an ideal with a basis $B = \{u, v\}$. Then, we have the following statements. \begin{enumerate}
		\item [(i)] A basis $B$ is good for twisting if and only if $F(u, v)  \le 0$, in which case the twisting matrix
		$T_\alpha$ is given by $\alpha=\left(\dfrac{\Im(v)^2-\Im(u)^2}{\Re(u)^2-\Re(v)^2}\right)^{\dfrac{1}{4}}_.$
		\item [(ii)] If $B$ is good for twisting, then $$\min\left\{\left(\Im (u^2)\right)^2,\left(\Im(v^2)\right)^2\right\}\le \vol^2\left(\Lambda_K(I)\right).$$
		\item[(iii)] A basis $\left\{u,v\right\}$ is good for twisting if and only if so is the basis $\left\{v,u\right\}$. Moreover, their well-rounded twisting lattices are similar.  
	\end{enumerate}
	\begin{proof}
		The two first statements are corollaries of Proposition \ref{theo2}. The last one can be implied by applying $(i)$ and Proposition \ref{theo5}.
	\end{proof}

\end{proposition}


From \eqref{(3)}, one obtains that $F(B)=F_1(B)F_2(B)$ where
 \begin{align}\label{(8)}
F_1(B)&=\dfrac{\overline{u}^2-u^2}{4i}+\dfrac{\overline{v}^2-v^2}{4i}+\dfrac{\overline{uv}-uv}{4i}=-\dfrac{1}{2}\left(\Im(u^2)+\Im(v^2)+\Im(uv)\right)\  \text{and}\\\nonumber
F_2(B)&=\dfrac{\overline{u}^2-u^2}{4i}+\dfrac{\overline{v}^2-v^2}{4i}-\dfrac{\overline{uv}-uv}{4i}=-\dfrac{1}{2}\left(\Im(u^2)+\Im(v^2)-\Im(uv)\right).
\end{align}

A similar version for Proposition 5 of \cite{damir2019well} is the following.

\begin{proposition}\label{prop6}
	Let $I \subset \mathcal{O}_K$ be an ideal. The hexagonal lattice is a twist of $\Lambda_K(I)$ if and only if $I$ has a basis $B = \{ u, v\}$ such that $F(B) = 0$.
\end{proposition}

\begin{proof}
	We use the fact that $F(u,v)=0$ if and only if $\Im(u^2)+\Im(v^2)=\pm \Im(uv)$. Equivalently, one has $\left|\cos \theta_{T_\alpha B}\right|=\dfrac{1}{2}$. In this case, the twist lattice of $\Lambda$ is hexagonal.
\end{proof}


Our goal is to compute all good bases of a given ideal lattice $\Lambda_K(I)$, up to similarity. Now we fix a suitable element $x\in I$ and find all good bases of the form $\{x, y\}$ of $I$.
\begin{definition}
	For $x\in I$, we say that $x$ can be extended to a (good) basis if there exists $y \in I$ 	such that $\{x, y\}$ is a (good) basis of $I$.
\end{definition}


Let $\{u,v\}$ be any basis of $I$ and write $x = au + cv$ for $a, c \in \mathbb{Z}$ without loss the generality, we can assume $a\ge0$. Then $x$ can be extended to a basis if and only if there exists $y = bu + dv \in I$ such that $ad - bc = \pm 1$. This occurs if and only if $(a,c)=1$. Combining with Proposition \ref{theo3} we obtain the following initial strategy to compute all well-rounded twists of a given ideal lattice $\Lambda_K(I)$, up to similarity.

\begin{enumerate}[$\square$]
	\item \textbf{Step 1:} Find a basis $\{u,v\}$ of $I$.
	\item \textbf{Step 2:} List all $x=au+cv\in I$ such that $\left(\Im(x^2)\right)^2\le \vol^2(\Lambda_K(I)),a\ge 0$ and $\gcd(a,c)= 1$. 
	\item \textbf{Step 3:} For each $x$ found in \textbf{Step 2}, we solve $F (x, y) \le 0$ for all possible $y$ such that $\{x, y\}$ is a basis of $I$. Note that such basis must satisfy the condition \eqref{condition}.
\end{enumerate}

\begin{remark}\label{rmk1}
	In \textbf{step 2}, we have to list  all elements $x$ such that $\left(\Im(x^2)\right)^2\le \vol^2(\Lambda_K(I))$. For example, $I=\mathcal{O}_K$ with $K=\mathbb{Q}(\sqrt{-5})$. We consider the elements $a$ and $c\sqrt{-5}$ ($a,c\in \mathbb{Z}$). The square of $a$ or $c\sqrt{-5}$ ($a,c\in \mathbb{Z}$) has zero imaginary part. It means the inequation $\left(\Im(x^2)\right)^2\le \vol^2\left(\Lambda_K(I)\right)$ may have infinitely many solutions. We can avoid this by using an idea given by Theorem \ref{theo5}.
\end{remark}


\begin{lemma}
	\label{bode1} Let $I\ne (0)$ be an integral ideal of $\mathcal{O}_K$. For all $x\ne 0$ and $x\in I$, we can choose a $\mathbb{Z}$-basis $B=\{u,v \}$ of $I$ such that $\Im(uv)$ and $\Im(vx)$ are non-zero.
\end{lemma}

\begin{proof}
	Fix the canonical basis $B=\{t,y+g\delta\}$ of $I$. Clearly, we have $\Im\left(t(y+g\delta)\right)\ne 0$. If $\Im(x)=0$, we choose $u=t, v=y+g\delta$ for $t,g> 0,y\ge 0$ such that $ y<t$, $g|t,y$ and $tg|N(y+g\delta)$. Then the basis $B'=\{u,v\}$ satisfies $\Im(uv)\ne 0$ and $\Im(vx)\ne 0$. If $\Im(x)\ne 0$, we choose $u=y+g\delta,v=t$, then $B'=\{u,v\}$ satisfies that $\Im(uv), \Im(vx)$ are non-zero.
\end{proof}

There are only finitely many $x\in I$ such that it can extend to a good basis. The proof of Lemma \ref{bode3} describes accurately a method to find all those elements.


\begin{lemma}
	\label{bode3} Suppose that $D$ is square-free and positive such that $-D\not \equiv 1\ (\text{mod  }4)$. Then there are finitely many elements $z\in I$ such that $z$ can extend to a good basis of $I$, up to similarity.
\end{lemma}

\begin{proof}
	Fix $\left\{ t,y+g\sqrt{-D}\right\}$ the canonical basis of $I$ where $ 0\le y<t$, $g|t,y$ and $tg|(y^2+g^2D)$. An arbitrary element of $I$ has the form $z=(at+cy)+cg\sqrt{-D}$ for some $a,c\in \mathbb{Z}$. By Proposition \ref{theo3}, we only consider the existence of an extendable basis $\{z,z'\}$ with $\left(\Im(z^2)\right)^2\le \vol^2(\Lambda_K(I))$. It is equivalent to $\left|c(at+cy)\right|\le \dfrac{t}{2}$. There are three cases.

	\begin{enumerate}[\textbf{Case } 1:]
		\item Solve the inequation $0<|c(at+cy)|\le \dfrac{t}{2}$. There are finitely many pairs $(a,c)$ satisfying the inequalities.
		\item  If $c=0$, then $z=at$. Thus $z$ can be extended to a good basis if and only if there exists $z'=(bt+dy)+dg\sqrt{-D}$ such that $\{z,z'\}$ is a good basis of $I$. It occurs when $ad=\pm 1$. In other words, one has $a=\pm 1$.
		\item  If $c\ne 0$ and $at+cy=0$, then $z$ can be extended to a good basis if there exists an element $(bt+dy)+dg\sqrt{-D}$ such that $\left|\begin{matrix}
		-\dfrac{cy}{t}&b\\c&d
		\end{matrix}\right|=\pm 1$. It is equivalent to that $c(dy+bt)=\pm t$. Therefore $c$ is a divisor of $t$. Because $t\ne 0$, there are finitely many such pairs $(a,c)$.	\end{enumerate}
\end{proof}


If $-D\equiv 1\ (\text{mod }4)$, the canonical basis of an ideal $I$ of $\mathcal{O}_K$ with $K=\mathbb{Q}(\sqrt{-D})$ is $t,y+g\dfrac{1+\sqrt{-D}}{2}$ where $ 0\le y<t$, $g|t,y$. We can write an arbitrary element of $I$ as $\dfrac{2at+2 cy+cg}{2}+\dfrac{cg\sqrt{-D}}{2}$. By Proposition \ref{theo3}, if this element can be extended to a good basis of $I$, then $|(2at+2cy+cg)c|\le t$. By using an argument similar to the one in the proof of Lemma \ref{bode3}, one obtains the following lemma.

\begin{lemma}
	\label{bode4} Suppose that $D$ is a square-free and positive integer such that $-D \equiv 1\ (\text{mod }4)$. Then there are finitely many elements $z\in I$ such that $z$ can be extended to a good basis of $I$, up to similarity.
\end{lemma}

From Lemma \ref{bode3} and Lemma \ref{bode4}, the following result is obtained.

\begin{theorem}\label{theo4}
	Let $I$ be a nonzero ideal of $\mathcal{O}_K$. Then there are finitely many elements in $I$ that can be extended to a good basis of $I$.
\end{theorem}


From the proof of Lemma \ref{bode3} and Lemma \ref{bode4}, we can list all elements of $I$ which can be extended to a basis of $I$. The proofs of two lemmas also yield an explicit method to find those elements. We replace \textbf{Step 2} of strategy (in page 6) by \textit{listing all elements of $I$ which can be extended to a basis of $I$}. The next result gives us a more efficient method to compute these bases.


\begin{lemma}\label{bode5}
	Let $I$ be an integral ideal with the basis $\{u,v\}$. Assume that an element $x=au+cv\in I$, that can be extended to a basis $B=\{x, y\}$ with $y=bu+dv$. Then we have the following. \begin{enumerate}[(i)]
		\item If $a=0$, then we can choose $x=v, y=u+dv$ and hence \begin{align}\label{(9)}
		F_1(B)=&-\dfrac{\Im(x^2)}{2}d^2+\left(-\Im(uv)- \dfrac{\Im(x^2)}{2} \right)d-\dfrac{\Im(x^2)+\Im(u^2)+\Im(uv)}{2}\\
		\label{(10)}F_2(B)=&-\dfrac{\Im(x^2)}{2}d^2+\left(-\Im(uv)+ \dfrac{\Im(x^2)}{2} \right)d-\dfrac{\Im(x^2)+\Im(u^2)-\Im(uv)}{2}.
		\end{align} 
		\item If $a\ne 0$ and $ad-bc=1$, then \begin{align}
		 \label{(11)} a^2F_1(B)=&-\dfrac{\Im(x^2)}{2}b^2-\left[a\left(\dfrac{2\Im(uv)}{2}+ \dfrac{\Im(x^2)}{2} \right)+2c\dfrac{\Im(v^2)}{2}\right]b\\\nonumber&-a^2\left(\dfrac{\Im(x^2)}{2}+ \dfrac{\Im(uv)}{2}\right)-\dfrac{\Im(v^2)}{2}(1+ ac)\\\label{(12)} a^2F_2(B)=&-\dfrac{\Im(x^2)}{2}b^2-\left[a\left(\dfrac{2\Im(uv)}{2}- \dfrac{\Im(x^2)}{2} \right)+2c\dfrac{\Im(v^2)}{2}\right]b\\\nonumber&-a^2\left(\dfrac{\Im(x^2)}{2}- \dfrac{\Im(uv)}{2}\right)-\dfrac{\Im(v^2)}{2}(1- ac).
		\end{align}
	\end{enumerate}
\end{lemma}

\begin{proof}
	\begin{enumerate}[(i)]
		\item From \eqref{(8)}, we have 
		\begin{align*}
		F_1(B)&=-\dfrac{1}{2}\left(\Im(x^2)+\Im(y^2)+\Im(xy)\right)=\dfrac{-\Im(x^2)}{2}-\dfrac{\Im((u+dv)^2)}{2}+\dfrac{\Im(v(u+dv))}{2}\\
		&=\dfrac{-\Im(x^2)}{2}-\dfrac{1}{2}\left(\dfrac{(u^2-\overline{u}^2)+d^2(v^2-\overline{v}^2)+2d(uv-\overline{uv})}{2i}\right)-\dfrac{1}{2}\dfrac{v(u+dv)-\overline{v(u+dv)}}{2i}\\
		&=\dfrac{-\Im(x^2)}{2}d^2-\left(\Im(uv)+\dfrac{\Im(x^2)}{2}\right)d-\dfrac{\Im(x^2)+\Im(u^2)+\Im(uv)}{2}.
		\end{align*}
		
		It is the result of \eqref{(9)}. By using a similar computation, we obtain the result in \eqref{(10)}.
		\item[(ii)] Here we have $ad-bc=1$, $d=\dfrac{1+bc}{a}$. Using equation \eqref{(8)} leads to the following.
		
	\end{enumerate}
	\begin{align*}
		a^2F_1(B)&=-\dfrac{a^2}{2}\left(\Im(x^2)+\Im(\left(bu+dv\right)^2)+\Im((au+cv)(bu+dv))\right)\\
		&=-\dfrac{a^2}{2}\left(\Im(x^2)+\Im\left(\left(bu+\dfrac{1+bc}{a}v\right)^2\right)+\Im\left((au+cv)\left(bu+\dfrac{1+bc}{a}v\right)\right)\right)\\
		&=-\dfrac{a^2}{2}\left[\Im(x^2)+\Im\left(bx+v\right)^2+a\Im\left((au+cv)(bx+v)\right)\right]\\
		&=-\dfrac{a^2\Im(x^2)}{2}-\dfrac{\Im(x^2)}{2}b^2-\dfrac{\Im(v^2)}{2}-\dfrac{2b\Im\left((au+cv)v\right)}{2}-\dfrac{ab\Im(x^2)}{2}-\dfrac{a^2\Im(uv)+ac\Im(v^2)}{2}\\
		&=-\dfrac{\Im(x^2)}{2}b^2-\left[a\left(\Im(uv)+\dfrac{\Im(x^2)}{2}\right)+c\Im(v^2) \right]b-a^2\left(\dfrac{\Im(x^2)}{2}+\dfrac{\Im(uv)}{2}\right)-\dfrac{\Im(v^2)}{2}(1+ac).
		\end{align*}
		Thus \eqref{(11)} is proved. The result in \eqref{(12)} can be obtained by using a similar computation.
\end{proof}


When $\Im(x^2)\ne 0$, the right sides of \eqref{(9)} and \eqref{(10)} are degree two polynomials in $d$ with the same discriminants. Indeed, we have (note that $v=x$)
\begin{align}
\nonumber \Delta_{F_1(B)}&=\left(-\Im(uv)-\dfrac{\Im(x^2)}{2}\right)^2-4\dfrac{\Im(x^2)}{2}\left(\dfrac{\Im(x^2)+\Im(u^2)+\Im(uv)}{2} \right)\\
\nonumber&=-\dfrac{3\left(\Im(x^2)\right)^2}{4}+\Im(uv)^2-\Im(x^2)\Im(u^2)=-\dfrac{3\Im(x^2)}{4}+\dfrac{(uv-\overline{uv})^2-(u^2-\overline{u}^2)(v^2-\overline{v}^2)}{(2i)^2}\\
\label{(13)}&=\left(\dfrac{u\overline{v}-v\overline{u}}{2i}\right)^2-\dfrac{3\left(\Im(x^2)\right)^2}{4}=\vol^2\left(\Lambda_K(I)\right)-\dfrac{3\left(\Im(x^2)\right)^2}{4}.
\end{align}

Similarly, one obtains that $F_2(B)$ has the same discriminant as of $F_1(B)$. 
Analogously, the discriminants of polynomials (in $b$) on the right side of \eqref{(11)} and \eqref{(12)}  are 
\begin{align}\label{(15)}
 &a^2\left(\vol^2\left(\Lambda_K(I)\right)-\dfrac{3\left(\Im(x^2)\right)^2}{4}\right).
\end{align}
The next result is an analogy to Theorem 3 in \cite{damir2019well}.


\begin{theorem}\label{theo5}
	Let $I$ be an ideal of $\mathcal{O}_K$ and let $x \in I$ such that $\left(\Im(x^2)\right)^2\le \vol^2(\Lambda_K(I)) $. Then $x$ can be extended to at most two good bases of $I$, up to similarity.
\end{theorem}

\begin{proof}
	Let us fix a basis $\{u, v\}$ of $I$ such that $\Im(uv)\ne 0$ and $\Im(vx)\ne 0$. Such a basis exists by Lemma \ref{bode1}. The condition $\Im(vx)\ne 0$ implies that $a\Im(uv)+c\Im(v^2)\ne 0$. Express $x = au+cv$ for some $a, c \in \mathbb{Z}$, and suppose that $y = bu + dv$ for some $b, d \in \mathbb{Z}$ and $\{x, y\}$
	is a good basis. We will employ the inequality $F (x, y) \le 0$ and the equality $ad - bc = \pm 1$ to solve for all possible $b$ and $d$. We consider two cases.
	\begin{itemize}
		\item \textbf{Case 1:} If $a=0$ , then $ad - bc = \pm 1$. It implies that $b = \pm 1$ and $c = \pm 1$, so $x = \pm v$ and $y = \pm u + dv$. By possibly replacing $x$ with $-x$ and $y$ with $-y$, which does not change the similarity class of the given basis, we may assume that our basis $\{x, y\}$ is of the form $\{x, y\} = \{v, u + dv\}$. We will show that there are at most two integers d such that this is a good basis.\\
		Let $f_i(d) = F_i(x, y)$ and $F_i(x, y)$ are as in \eqref{(9)} and \eqref{(10)}. By Proposition \ref{theo2}, we must find all $d$ such that $f_1(d)$ and $f_2(d)$ have opposite signs, or such that at least one of them are zero. Using \eqref{(9)} and \eqref{(10)}, we divide our proof into 2 cases.
		
		\begin{itemize}
			\item \textbf{Case 1.1.} If $\Im(x^2)=0$, the functions in \eqref{(9)} and \eqref{(10)} become \begin{align*}
	    f_1(d)&=-\Im(uv)d-\dfrac{\Im(u^2)+\Im(uv)}{2} \ \text{and}\\
		f_2(d)&=-\Im(uv)d-\dfrac{\Im(u^2)-\Im(uv)}{2}.
			\end{align*}
			There are at least one of $f_1(d)$ and $f_2(d)$ which are equal to zero if and only if $d=\beta_1$ or $d= \beta_2$ where
			\begin{align}
			\label{(18)}\beta_1=\dfrac{\Im(u^2)+ \Im(uv)}{-2\Im(uv)} \text{ and }	\beta_2=\dfrac{\Im(u^2)- \Im(uv)}{-2\Im(uv)}.
			\end{align} 
			In addition, $f_1(d)$ and $f_2(d)$ have opposite signs if $d\in \left(\beta_1, \beta_2\right)$. Therefore, $d\in\left[\beta_1,\beta_2\right]$. However, $\beta_2-\beta_1=1$, there are at most two values of $d$ satisfying the condition.
			
			\item \textbf{Case 1.2.} If $\Im(x^2)\ne 0$, the functions in \eqref{(9)} and \eqref{(10)} are second degree polynomials with the same discriminants $\Delta=\vol^2\left(\Lambda_K(I)\right)-\dfrac{3\left(\Im(x^2)\right)^2}{4}$ by \eqref{(13)}. The polynomial $f_1(d)$ has two roots 
			\begin{align}\label{(20)}
			\beta_{11}=\dfrac{-\left(-\Im(uv)- \dfrac{\Im(x^2)}{2} \right)+ \sqrt{\Delta}}{-\Im(x^2)},\beta_{12}=\dfrac{-\left(-\Im(uv)- \dfrac{\Im(x^2)}{2} \right)- \sqrt{\Delta}}{-\Im(x^2)}.
			\end{align} The polynomial $f_2(d)$ has two roots 
			\begin{align}
			\label{(22)}
			\beta_{21}=\dfrac{-\left(-\Im(uv)+\dfrac{\Im(x^2)}{2} \right)+ \sqrt{\Delta}}{\Im(x^2)},\beta_{22}=\dfrac{-\left(-\Im(uv)+\dfrac{\Im(x^2)}{2} \right)- \sqrt{\Delta}}{-\Im(x^2)}.
			\end{align}
			There are at least one of $f_1(d)$ and $f_2(d)$ which are equal to zero if and only if $d\in\left\{\beta_{11},\beta_{12},\beta_{21},\beta_{22}\right\}$. In these cases, we have $F(B)=f_1(d) f_2(d)=0$, then one obtains four hexagonal twist lattices (by Proposition \ref{prop6}). Therefore, they are all similar. Thus there is at most one $d$, up to similarity. In addition, $f_1(d)$ and $f_2(d)$ have opposite signs if and only if $d\in \left(\beta_{11}, \beta_{21}\right)=J_1$ or $d\in \left(\beta_{12},\beta_{22}\right)=J_2$. These open intervals have width one. As a result they only contain at most one integer. Hence, there are at most two $d$ satisfying the condition.    	
		\end{itemize}

		\item \textbf{Case 2:} If $a\ne 0$, then $d=\dfrac{1+bc}{a}$. Multiplying by $-1$ if necessary
		we may assume $a > 0$. We again explicitly compute all $y = bu + dv \in I$ such that $\{x, y\}$ is a good basis of I, up to similarity. By possibly replacing $y$ with $-y$ we may assume $ad -bc = 1$, and solve for $d$ in terms of $b$ as $d = \dfrac{1 + bc}{a}$. Setting $f_i(b) = F_i(x, y)$, we
		wish to find all integers $b$ such that $d = \dfrac{1 + bc}{a} \in \mathbb{Z}$, and that either $f_1(b)$ and $f_2(b)$ have opposite signs or such that at least one of them are zero. From \eqref{(11)} and \eqref{(12)}, one can consider two cases as below.
	
		\begin{itemize}
			\item \textbf{Case 2.1.} If $\Im(x^2)=0$, the functions in \eqref{(11)} and \eqref{(12)} become 
			\begin{align*}
				a^2f_1(b)=&-\left(a\Im(uv)+c\Im(v^2)\right)b-a^2\left(+ \dfrac{\Im(uv)}{2}\right)-\dfrac{\Im(v^2)}{2}(1+ ac) \text{ and}\\
			a^2f_2(b)=&-\left(a\Im(uv)+c\Im(v^2)\right)b-a^2\left(- \dfrac{\Im(uv)}{2}\right)-\dfrac{\Im(v^2)}{2}(1- ac).
			\end{align*}
			The condition $a\Im(uv)+c\Im(v^2)\ne 0$ follows that $a^2f_1(b)$ and $a^2f_2(b)$ are linear polynomials in $b$. The roots of $a^2f_1(b)$ and $a^2f_2(b)$ are respectively
			\begin{align}\label{(23)}
			    \beta_1=\dfrac{a^2\left( \dfrac{\Im(uv)}{2}\right)+\dfrac{\Im(v^2)}{2}(1+ ac)}{-\left(a\Im(uv)+c\Im(v^2)\right)}\text{ and }\beta_2=\dfrac{a^2\left( \dfrac{-\Im(uv)}{2}\right)+\dfrac{\Im(v^2)}{2}(1- ac)}{-\left(a\Im(uv)+c\Im(v^2)\right)}.
			\end{align} 
			There are at least one of $f_1(b)$ and $f_2(b)$ which are equal to zero if and only if $b=\beta_1$ or $b=\beta_2$. In addition, $a^2f_1(b)$ and $a^2f_2(b)$ have opposite signs if and only if $b \in [\beta_1,\beta_2]=J$. The interval $J$ has width $a$, so the equation $bc+1\equiv 0 \mod a$ has at most one solution in $J$. Therefore, there are at most two pairs $(b,d)$ satisfying the above condition as we expect.
			
			\item \textbf{Case 2.2.} If $\Im(x^2)\ne 0$, the functions in \eqref{(9)} and \eqref{(10)} are second degree polynomials with the same discriminants $\Delta=a^2\left(\vol^2\left(\Lambda_K(I)\right)-\dfrac{3\left(\Im(x^2)\right)^2}{4}\right)$ by \eqref{(15)}. Then the polynomial $a^2f_1(b)$ has two roots 
			\begin{align}
			\label{(28)}\beta_{11}&=\dfrac{\left[a\left(\Im(uv)+\dfrac{\Im(x^2)}{2}\right)+c\Im(v^2)+ \sqrt{\Delta}\right]}{-\Im(x^2)}\text{ and }\\\nonumber
			\beta_{12}&=\dfrac{\left[a\left(\Im(uv)+\dfrac{\Im(x^2)}{2}\right)+c\Im(v^2)- \sqrt{\Delta}\right]}{-\Im(x^2)}.
			\end{align}
			The polynomial $a^2f_2(b)$ has two roots 
			\begin{align}
			\label{(30)}\beta_{21}&=\dfrac{\left[a\left(\Im(uv)-\dfrac{\Im(x^2)}{2}\right)+c\Im(v^2)+ \sqrt{\Delta}\right]}{-\Im(x^2)}\text{ and}\\\nonumber
			\beta_{22}&=\dfrac{\left[a\left(\Im(uv)-\dfrac{\Im(x^2)}{2}\right)+c\Im(v^2)- \sqrt{\Delta}\right]}{-\Im(x^2)}.
			\end{align}
			There are at least one of $a^2f_1(b)$ and $a^2f_2(b)$ which are equal to zero if and only if $b\in\left\{\beta_{11},\beta_{12},\beta_{21},\beta_{22}\right\}$. In these cases, we have four hexagonal twist lattices (by Proposition \ref{prop6}) and therefore, they are similar. Moreover, there is at most one $d$, up to similarity. In addition, $a^2f_1(b)$ and $a^2f_2(b)$ have opposite signs if and only if $d\in \left(\beta_{11}, \beta_{21}\right)=J_1$ or $d\in \left(\beta_{12},\beta_{22}\right)=J_2$. This open intervals have width $a$, so they contain at most one integer which is a solution of $bc+1\equiv 0\ \mod a$. Hence, there are at most two expected pairs $(b,d)$. 
		\end{itemize}
	\end{itemize}
\end{proof}


If  $-D\not\equiv 1 \mod 4$, the ring of integers $O_K$ of  $K=\mathbb{Q}(\sqrt{-D})$ has the canonical basis $\{u=1,v=\sqrt{-D}\}$. Using the proof of Theorem \ref{theo5}, one can show that $O_K$ only has the following good bases $x = au+cv$, $y = bu + dv$ for $(a, c, b, d) \in \{(1,0,0,1),(1,0,0,-1),(0,1,1,0)\}$. Since all well-rounded lattices defined by these bases are similar, one obtains that $O_K$ has only one well-rounded twist up to similarity. We have the following corollary that is an analogy with the result of Corollary 3 in \cite{damir2019well}.

\begin{corollary}\label{uniqueWR}
Let $D$ be a square-free integer such that $-D\not \equiv 1 \pmod 4$, and let $K=\mathbb{Q}(\sqrt{-D})$. Then the lattice $\Lambda_K$ has a unique well-rounded twist, which is an orthogonal lattice, up to similarity.
\end{corollary}
\begin{proof}
The canonical basis of $\mathcal{O}_K$ is $\left\{1,\sqrt{-D}\right\}$ and $\vol(\Lambda_K)=\sqrt{D}$. Suppose that $x=a+c\sqrt{-D}$ can be extended to a good basis. By Proposition \ref{theo3}, it is sufficient to consider the case in which  $|ac|\le \dfrac{1}{2},(a,c)=1$ and $a\ge 0$. It implies $(a,c)\in \left\{(0,1),(0,-1),(1,0)\right\}$.
\begin{itemize}
    \item If $(a,c)=(0,1)$ then $x=\sqrt{-D}$. By Theorem \ref{theo5}, the basis to which $x$ extends is $\{\sqrt{-D},1 \}$. The well-rounded twist lattice of this basis is orthogonal.
    \item If $(a,c)=(0,-1)$, the result is the same with the case $(a,c)=(0,1)$. 
    \item If $(a,c)=(1,0)$ then $x=1$, by Theorem \ref{theo5}, the basis to which $x$ extends is $\{1,\sqrt{-D} \}$. Thus the well-rounded twist lattice of this basis is orthogonal.
\end{itemize}
Therefore, for all cases,  $O_K$ has a unique well-rounded twist which is similar to an orthogonal lattice, up to similarity.
\end{proof}

\begin{corollary}
\label{coro2} Let $D$ be a square-free integer such that $-D \equiv 1 \pmod 4$, and let $K=\mathbb{Q}(\sqrt{-D})$.  Then the lattice $\Lambda_K$ has a unique well-rounded twist, which is a hexagonal lattice, up to similarity.
\end{corollary}

\begin{proof}
The canonical basis of $\mathcal{O}_K$ is $\left\{1,\dfrac{1+\sqrt{-D}}{2}\right\}$ and $\vol(\Lambda_K)=\dfrac{\sqrt{D}}{2}$. Suppose that $x=a+c\left(\dfrac{1+\sqrt{-D}}{2}\right)$ can be extended to a good basis. By Proposition \ref{theo3}, it is sufficient to consider the case in which  $|(2a+c)c|\le 1,(a,c)=1$ and $a\ge 0$. It implies $(a,c)\in \left\{(1,0),(1,-2),(0,1),(1,-1)\right\}$. By using a similar argument as the one of Theorem \ref{theo5}, we can compute all tuples $(a,c,b,d)$ such that $F(au+cv, bu+dv)\le 0$, where $u=1, v=\dfrac{1+\sqrt{-D}}{2}$. One can easily checks that the value of $F$ at all bases are equal to zero. Therefore, these bases have only one  well-rounded twist lattice, which is hexagonal, up to similarity.
\end{proof}

\begin{remark}
    Our results on well-rounded twist lattices of the ring of integers $O_K$ (Corollaries \ref{uniqueWR} and \ref{coro2}) are more general compared to Lemma 2.2 in \cite{FukIdeal} which states that the lattices $O_K$ is well-rounded (without twisting) if and only if $D=1,3$. Indeed,  when $D=1$ then $O_K =\mathbb{Z}[i]$ which  is well-rounded and orthogonal, and is a particular case of Corollary \ref{uniqueWR}. When $D=3$, then $O_K = \mathbb{Z}\left[\dfrac{1+\sqrt{-3}}{2}\right]$ which is well-rounded and hexagonal, and is a particular case of Corollary \ref{coro2}.
\end{remark}

Now let $I$ be an integral ideal of $K= \mathbb{Q}(\sqrt{-D})$ with the canonical basis $\left\{t,y+g\delta \right\}$. In case $y \ne 0$, we can easily apply a similar argument as in the proof of Theorem \ref{theo5} to find upper bounds for the number of well-rounded twists of $I$ which are presented in Corollaries 3 and 4 as below. Note that these results may not be true for real quadratic fields and a similar result has not been proved in \cite{damir2019well}.

\begin{corollary}\label{numbertwists}
Let $D$ be a squarefree integer with $-D\not \equiv 1 \pmod 4$ and let  $I$ be an ideal of $\mathbb{Q}(\sqrt{-D})$ with the canonical basis  $\left\{t,y+g\delta \right\}$. Then $I$ has at most  $6+2\left[\dfrac{y+1}{2}\right]$ well-rounded twists.
\end{corollary}

\begin{proof}
The result can be easily obtained from counting the number solutions of the inequations $\left|(at+cy)c\right|\le \dfrac{t}{2}, a\ge 0,$ and $(a,c)=1$ and by applying Theorem \ref{theo5}.
\end{proof}
Moreover, Corollary \ref{numbertwists} can be implied immediately from Algorithm \ref{algo1}. Similarly, one has the following result when $-D\equiv 1 \pmod 4$.

\begin{corollary}
Let $D$ be a squarefree integer with where $-D \equiv 1 \pmod 4$ and let $I$ be an ideal of  $\mathbb{Q}(\sqrt{-D})$ with the canonical basis $\left\{t,y+g\delta \right\}$. Then $I$ has at most  $6+2\left[\dfrac{2y+g+1}{2}\right]$ well-rounded twists.
\end{corollary}


\begin{proposition}
    \label{theo6}
	Every ideal of $\mathcal{O}_K$ has at least one well-rounded twist lattice.
\end{proposition}
\begin{proof}
We prove this theorem for the case $-D\not \equiv 1 \mod 4$, the case $-D \equiv 1 \mod 4$ can be proved using a similar argument. \\
	As in the proof of Lemma \ref{bode3}, one can see that $z=1.t+0.(y+g\sqrt{-D})=t$ is an element of $I$ which can be extended to a good basis. Moreover, since $\Im(z^2)=0$, by setting $u=t$ and $v=y+g\sqrt{-D}$ and using \eqref{(23)}, one has $\beta_1=\dfrac{t+2y}{-2t}$ and $\beta_2=\dfrac{-t+2y}{-2t}$. 
	There are at most two integers and at least one integers in the interval $[\beta_1, \beta_2$]. Using an argument similar to the one in the proof of Theorem \ref{theo5}, we obtain a good basis of $I$. It provides that $I$ has a well-rounded twist lattice.
\end{proof}
\begin{remark}
In \cite{solan2019stable}, it is proved that every lattice (in any dimension) has at least a well-rounded twist, thus, Proposition \ref{theo6} can be considered as a particular case of the mentioned result. Our proof, however, uses an independent  argument and result from the ones in \cite{solan2019stable}. Indeed Proposition \ref{theo6} is implied from the proofs of Lemma \ref{bode3} and Theorem \ref{theo5}. We remark that a similar result has not been proved in \cite{damir2019well} for real quadratic fields. 
\end{remark}

%% file: alg_example.tex

The proof of Theorem \ref{theo5} gives us explicit formulae (see \ref{(18)}, \ref{(20)}, \ref{(22)}, \ref{(23)}, \ref{(28)} and \ref{(30)}) to compute all well-rounded twists of an ideal $I$ given by any $\mathbb{Z}$-basis $\{u, v\}$ of $I$. In practice, $I$ is given by the canonical basis that can be efficiently computed if two generators of $I$ over $O_K$ are provided. We also note that the conditions $\Im(uv)\ne 0$ and $\Im(vx)\ne 0$ in Theorem \ref{theo5} is necessary only if $\Im(x^2)=0$. We can exchange the two vectors in the canonical basis of $I$ to have these conditions. Moreover, the canonical basis provides us simpler formulae for $\beta_1,\beta_2,\beta_{11},\beta_{12},\beta_{21},\beta_{22}$ than the ones in a general case shown as below.

 In case $-D\not \equiv 1\  (\text{mod } 4)$, then  $\{t,y+g\sqrt{-D}\}$ is the canonical basis of $I$. Let $u=t,v=y+g\sqrt{-D}$.

\begin{enumerate}[1.]
	\item If $a=0,\Im(x^2)=0$, from $ad-bc=\pm 1$, one has $c=\pm 1$. In these cases, since $\Im (uv)=tg\sqrt{D}\ne 0$, equations in \eqref{(18)} become $\beta_1=\dfrac{-1}{2},\beta_2=\dfrac{1}{2}.$ It implies that $d=0$ and $b=\pm 1$. We only receive the tuple $(0,1,1,0)$ as others give the same well-rounded twist lattices, up to similarity.  
	
	\item If $a=0,\Im(x^2)\ne 0$, then it implies $c^2=1$. Let $\alpha=\dfrac{t}{2y}$. Then  \eqref{(20)} and \eqref{(22)} become
	\begin{align}\label{(32)}
	\beta_{11}&=
	-\dfrac{1}{2}-\alpha-\sqrt{\alpha^2-\dfrac{3}{4}}, 
	\hspace*{0.5cm}\beta_{21}=\beta_{11}+1\\
	\label{(33)}\beta_{12}&=-\dfrac{1}{2}-\alpha+\sqrt{\alpha^2-\dfrac{3}{4}}, \hspace*{0.5cm}\beta_{22}=\beta_{12}+1.
	\end{align}

	\item If $a\ne 0,\Im(x^2)=0$ and $c=0$, the equation \eqref{(23)} becomes\begin{align}\label{(20a)}
	    \beta_1=\dfrac{-1}{2}-\dfrac{y}{t}, \beta_2=\beta_1+1,
	\end{align} where $u=t,v=y+g\sqrt{-D}$. If $a\ne 0,\Im(x^2)=0$ and $c\ne 0$, the equation \eqref{(23)} becomes \begin{align}
	    \label{(20b)}\beta_1=-\dfrac{a}{2},\beta=\dfrac{a}{2},
	\end{align}  where $u=y+g\sqrt{-D},v=t$.
	
	\item If $a\ne 0,\Im(x^2)\ne 0$, denote by $\beta=\dfrac{t}{2c(at+cy)}$, then 
	\eqref{(28)},\eqref{(30)} become
	\begin{align}\label{(36)}
	\left\{\beta_{11},\beta_{12} \right\}=\left\{\dfrac{-ac-2}{2c}+a\beta\pm a\sqrt{\beta^2-\dfrac{3}{4}}\right\},\\\label{(37)}
	\left\{\beta_{21},\beta_{22} \right\}=\left\{\dfrac{ac-2}{2c}+a\beta\pm a\sqrt{\beta^2-\dfrac{3}{4}}\right\}.
	\end{align}
\end{enumerate}

In the case $-D \equiv 1\  (\text{mod } 4)$, then $\left\{t,y+g\dfrac{1+\sqrt{-D}}{2}\right\}$ is the canonical basis of $I$. Let $u=t,v=y+g\dfrac{1+\sqrt{-D}}{2}$.

\begin{enumerate}[1.]
	\item If $a=0,\Im(x^2)=0$, it implies that $\dfrac{2y+g}{2}g\sqrt{D}=0$, which cannot happen.
	
	\item If $a=0,\Im(x^2)\ne 0$, then $c^2=1$. Let $\alpha=\dfrac{t}{2y+g}$. Then \eqref{(20)},\eqref{(22)} become 
	\begin{align}
	\label{(38)}\beta_{11}&=-\dfrac{1}{2}-\alpha-\sqrt{\alpha^2-\dfrac{3}{4}}, \beta_{21}=\beta_{11}+1\\
	\label{(39)}\beta_{12}&=-\dfrac{1}{2}-\alpha+\sqrt{\alpha^2-\dfrac{3}{4}}, \beta_{22}=\beta_{12}+1.
	\end{align}

	\item If $a\ne 0,\Im(x^2)=0$ and $c=0$, then \eqref{(23)} becomes
	\begin{align}\label{(40a)}
	\beta_1=-\dfrac{1}{2}-\dfrac{2y+g}{2t}\text{ and }\beta_2=\beta_1+a.
	\end{align}
	where $u=t,v=y+g\dfrac{1+\sqrt{-D}}{2}$. If $a\ne 0,\Im(x^2)=0$ and $c\ne 0$, then \eqref{(23)} becomes
	\begin{align}\label{(40b)}
	    \beta_1=-\dfrac{a}{2},\beta_2=\dfrac{a}{2}.
	\end{align} where $u=y+g\dfrac{1+\sqrt{-D}}{2},v=t$.
	\item If $a\ne 0,\Im(x^2)\ne 0$, denote by $\beta=\dfrac{t}{c(2at+2cy+cg)}$, then  \eqref{(28)},\eqref{(30)} become 
	\begin{align}\label{(42)}
	\left\{\beta_{11},\beta_{12} \right\}=\left\{\dfrac{-ac-2}{2c}+a\beta\pm a\sqrt{\beta^2-\dfrac{3}{4}}\right\},\\\label{(43)}\left\{\beta_{21},\beta_{22} \right\}=\left\{\dfrac{-ac+2}{2c}+a\beta\pm a\sqrt{\beta^2-\dfrac{3}{4}}\right\}.
	\end{align}
\end{enumerate}

Finally, we have a more efficient strategy to find all good bases of an ideal lattice $I$ compared to the one in page 6 as follows.
\begin{enumerate}[$\square$]
	\item \textbf{Step 1*:} Find the canonical basis $\{u,v\}=\{t,y+g\delta\}$ of $I$.
	\item \textbf{Step 2*:} List all elements of $I$ which can be extended to a basis of $I$.
	\item \textbf{Step 3*:} For each $x$ found in \textbf{Step 2}, identifying all good bases $\{x,y\}$ by using the formulae above. Note that such basis must satisfy the condition \eqref{condition}.
\end{enumerate}

In \textbf{Step 2*}, in case $-D \not \equiv 1 \mod 4$, we want to find $x=au+cv=(at+cy)+cg\sqrt{-D}$ which can be extended to a good basis of I. It is equivalent to $|(at+cy)c|\le \dfrac{t}{2} $. Lemma \ref{bode3} provides us an idea to list all pairs $(a,c)$.\\
In \textbf{Case 1}, the inequality $0<|(at+cy)c|\le \dfrac{t}{2} $ implies $0<|c|\le \dfrac{t}{2}$  and $a\le \dfrac{y+1}{2}$. If $a=0$, then $c=\pm1$ and we also have $0<y\le \dfrac{t}{2}$. For each $a\in \left[1,\dfrac{y+1}{2}\right]$, we must find $c$ satisfing that \begin{align}
    \label{a}-\dfrac{t}{2}\le \left(at+cy\right)c\le \dfrac{t}{2}.
\end{align}
Let $\alpha=\dfrac{t}{2y}\left(\alpha>0\right)$, by considering \eqref{a} as the inequation system in $c$, we obtain\begin{align}
\label{(45)}    -a\alpha-\sqrt{a^2\alpha^2+\alpha}\le c\le -a\alpha -\sqrt{a^2\alpha^2-\alpha}\\
 \label{(46)}-a \alpha+\sqrt{a^2\alpha^2-\alpha}   \le   c   \le   -a \alpha+\sqrt{a^2\alpha^2+\alpha}.
\end{align}
Since $-1<-a \alpha+\sqrt{a^2\alpha^2-\alpha} $ and $-a \alpha+\sqrt{a^2\alpha^2+\alpha}<1$ for all $a\ge 1$, the inequality \eqref{(46)} implies that $c=0$, it is contradict to $|c|>0$. Moreover, if $a\ge 2$, then $\left( -a\alpha -\sqrt{a^2\alpha^2-\alpha}\right)-\left(-a\alpha-\sqrt{a^2\alpha^2+\alpha}\right)< 1$ and if $a=1$, then $\left( -a\alpha -\sqrt{a^2\alpha^2-\alpha}\right)-\left(-a\alpha-\sqrt{a^2\alpha^2+\alpha}\right)\le \sqrt{2}$. Therefore, for each $a\ge 2$, we get at most one $c$ and when $a=1$, we have at most two c.


 In \textbf{Case 3}, if $y=0$, then $a=0$ and $c=\pm 1$. If $y\ne 0$, one has the conditions $at+cy=0$ and $\gcd(a,c)=1$. Thus, $c=-\dfrac{t}{\gcd(t,y)}$ and $a=\dfrac{y}{\gcd(t,y)}$.
 
Similarly, in the case $-D\equiv 1 \mod 4$, one obtains $a\le \dfrac{2y+g+1}{2}$
and then chooses nonzero $c$ satisfying \eqref{(45)}, $2at+2cy+g\ne 0$ and $\gcd(a,c)=1$ where $\alpha=\dfrac{t}{2y+g}$.
Moreover, the conditions $2at+2cy+cg=0$ and $\gcd(a,c)=1$ imply that $a=\dfrac{2y+g}{\gcd(2t,2y+g)}$ and  $c=-\dfrac{2t}{\gcd(2t,2y+g)}$.\\


In the case $-D\not \equiv 1\ (\text{mod } 4)$, a good basis of $I$ has a following form $$\left\{at+c(y+g\sqrt{-D}), bt+d(y+g\sqrt{-D})\right\}.$$ Condition \eqref{condition} can be rewritten as follow.
\begin{align}\label{condition2}
    (at+cy)d+(bt+dy)cg \ne 0.
\end{align}
We compute all good bases of $\Lambda_K(I)$ by the following algorithm.

\begin{algorithm}
	\label{algo1}
	(For $-D\not\equiv 1 \mod 4$)
	\begin{enumerate}[$\bullet$]
	\item\textbf{Input:} $D,t,y,g$ where $\left\{t,y+g\sqrt{-D}\right\}$ is the canonical basis of $I$.
	\item 	\textbf{Output:} The list $L$ of all tuples $\left(a,c,b,d\right)$ where $\left\{at+c(y+g\sqrt{-D}), bt+d(y+g\sqrt{-D})\right\}$ is a good basis of $I$, up to similarity.
	\end{enumerate}
	
	\begin{enumerate}
		\item[\textbf{Step 1}:] Add $(1,0,b,1)$
 into $L$ where $b\in\left[-\dfrac{1}{2}-\dfrac{y}{t},\dfrac{1}{2}-\dfrac{y}{t}\right]$.
 
		\item[\textbf{Step 2}:] If $y=0$, then add $(0,1,1,0)$ into $L$.
		
		\item[\textbf{Step 3}:] If $y\ne 0$, then
		\begin{enumerate}[3.1.]
		\item Compute $c=-\dfrac{t}{\gcd(t,y)}$ and $a=\dfrac{y}{\gcd(t,y)}$, then replace $\{a,c\}$ with $\{-c,-a\}$. Using \eqref{(20b)} to compute $\beta_1,\beta_2=\beta_1+a$. Add $(a,c,-b,-d)$ satisfies \eqref{condition2} into $L$ where $d\in[\beta_1,\beta_2]$ such that $1+dc$ is a multiple of $a$ and $b=\dfrac{1+cd}{a}$.
		
			\item If $t\ge 2y$, then compute $\alpha=\dfrac{t}{2y}$ and compute $\beta_{11},\beta_{12}$ using \eqref{(32)},\eqref{(33)}. Let $\beta_{21}=\beta_{11}+1$ and $\beta_{22}=\beta_{12}+1$. Add all tuples $(0,1,1,d)$ satisfies \eqref{condition2} into $L$ where $d\in [\beta_{11},\beta_{21}]\cup[\beta_{12},\beta_{22}]$.
		 For each integer a in $\left[1,\dfrac{y+1}{2}\right]$, compute all nonzero integers $c$ satisfying \eqref{(45)} and $at+cy\ne 0$ and $\gcd(a,c)=1$. Compute $\beta=\dfrac{t}{2c(at+cy)}$ and $\beta_{11},\beta_{12}$ by using \eqref{(36)}. Let $\beta_{21}=\beta_{11}+a,\beta_{22}=\beta_{12}+a$. Add $(a,c,b,d)$ satisfies \eqref{condition2} into $L$ where $b\in[\beta_{11},\beta_{21}]\cup[\beta_{12},\beta_{22}]$ satisfying that $1+bc$ is a multiple of $a$ and $d=\dfrac{1+bc}{a}$.
		 
		\item If $t<2y$, for each integer a in $\left[1,\dfrac{y+1}{2}\right]$, compute all nonzero integers $c$ satisfying \eqref{(45)}, $at+cy\ne 0$ and $\gcd(a,c)=1$. Compute $\beta=\dfrac{t}{2c(at+cy)}$ and $\beta_{11},\beta_{12}$ by using \eqref{(36)}. Let $\beta_{21}=\beta_{11}+a$ and $\beta_{22}=\beta_{12}+a$. Add $(a,c,b,d)$ satisfies \eqref{condition2} into $L$ where $b\in[\beta_{11},\beta_{21}]\cup[\beta_{12},\beta_{22}]$ satisfying that $1+bc$ is a multiple of $a$ and $d=\dfrac{1+bc}{a}$.
		\end{enumerate} 
	\end{enumerate}
\end{algorithm}

\begin{example}
	Consider $K=\mathbb{Q}(\sqrt{-201})$ and $I=\langle6+3\sqrt{-201}\rangle$ an ideal of $K$. We will find all good bases of $I$ as follow.
\end{example}
The canonical basis of $I$ is $\left\{615,6+3\sqrt{-201}\right\}$. 
Here $D=-201,t=615,y=6,g=3$. We follow all the steps of Algorithm \ref{algo1} as below.
\begin{enumerate}
\item[\textbf{Step 1:}] Since $b \in \left[\dfrac{-209}{410},\dfrac{201}{410}\right]$, we have $b=0$. Add $(1,0,0,1)$ into $L$. 

\item[\textbf{Step 2:}] We ignore Step 2 since $y=6\ne 0$.

\item[\textbf{Step 3:}]
\begin{enumerate}[3.1.]
	\item We have $c=-\dfrac{615}{\gcd(615,6)}=-205$ and $a=\dfrac{6}{\gcd(615,6)}=2$. Replace $(a,c)$ with $(-c,-a)$, one has $a=205,c=-2$. Then using \eqref{(20b)}, one obtains $\beta_{1}= -102.5,\beta_2= 102.5$. We choose $d\in[\beta_1,\beta_2]$ such that $\dfrac{1-2d}{205}$ is an integer. Thus $d=-102$ and hence $b=1$. We add $(2,-205,-1,102)$ into $L$.
\item	Since $t\ge 2y$, one obtains $$\alpha=\dfrac{t}{2y}=\dfrac{205}{4},\beta_{11}\approx -102.99,\beta_{21}\approx -101.99,\beta_{12}\approx -0.5,\beta_{22}\approx 0.5.$$ Then, $d\in[-102.99,-101.99]\cup[-0.5,0.5]$. Thus $d=-102$ or $d=0$. 
Hence, we add $(0,1,1,-102),(0,1,1,0)$ into $L$.\\
Since $a\in \left[1,\dfrac{7}{2}\right]$, then $a\in \{1,2,3\}$.
\begin{itemize}
	\item When $a=1$, one implies $c=-102$ that satisfies $\eqref{(45)}$ and $\gcd(a,c)=1$. Then
	$\beta=\dfrac{-205}{204},\beta_{11}\approx -0.99,\beta_{21}\approx 0.01,\beta_{12}\approx -2.01,\beta_{22}\approx -1.01.$\\
	It implies $b\in \{0,-2\}$ and $d\in \{1,205\}$, respectively.
	\item When $a=2$, then $c=-205$.	Since $at+cy=2.615-205.6=0$, then we eliminate the pair $(2,-205)$.
	\item When $a=3$, there is no value $c$ satisfying \eqref{(45)}.
\end{itemize}
\end{enumerate}
\end{enumerate}

Thus, $L$ contains there are 6 tuples listed the following table of which each column contains tuples defining the same lattice. 
\begin{center}
    \begin{tabular}{|c|c|c|}
     (1,0,0,1)&(2,-205,-1,102)&(0,1,1,-102)\\
     (0,1,1,0) &(1,-102,-2,205)&(1,-102,0,1)
\end{tabular}.
\end{center}
Therefore, there are 3 well-rounded twists of $I$, up to similarity, defined by the following tuples $\left\{\begin{matrix}(1,0,0,1),(2,-205,-1,102),(0,1,1,-102)\end{matrix}\right\}.$


\begin{remark}\label{similarity}
Once can easily checks the similarity of well-rounded twists defined by tuples $(a, b, c, d)$ in $L$ obtained from above algorithms by computing $|\cos \theta_{T_\alpha B}|$ (see Lemma \ref{similarwr}). Indeed,  one has $|\cos \theta_{T_\alpha B}|=\left|\dfrac{(at+cy)c+(bt+dy)d}{(at+cy)d+(bt+dy)c}\right|$ in case $-D\not\equiv 1 \mod 4$ and \\
$|\cos \theta_{T_\alpha B}|=\left|\dfrac{\left(at+c.\dfrac{2y+g}{2}\right)c+\left(bt+d.\dfrac{2y+g}{2}\right)d}{\left(at+c.\dfrac{2y+g}{2}\right)d+\left(bt+d.\dfrac{2y+g}{2}\right)c}\right|$ 
 in case $-D\equiv 1 \mod 4$.
\end{remark}

In the case $-D \equiv 1\ (\text{mod } 4)$, a good basis of $I$ has a following form $$\left\{at+c\left(y+g\dfrac{1+\sqrt{-D}}{2}\right),	 bt+d\left(y+g\dfrac{1+\sqrt{-D}}{2}\right)\right\}.$$ Condition \eqref{condition} can be rewritten as follow.
\begin{align}\label{condition3}
    \left(at+c.\dfrac{2y+g}{2}\right)d+\left(bt+d.\dfrac{2y+g}{2}\right)c\ne 0.
\end{align}
We compute all good bases of $\Lambda_K(I)$ by the following algorithm.
\begin{algorithm}\label{algo2}(For $-D\equiv 1\mod 4$)
	\begin{enumerate}[$\bullet$]
	\item \textbf{Input:} $D,t,y,g$ where $\left\{t,y+g\dfrac{1+\sqrt{-D}}{2}\right\}$ is the canonical basis of $I$.
	\item \textbf{Output:} The list $L$ of all tuples $\left(a,c,b,d\right)$ where\\ $\left\{at+c\left(y+g\dfrac{1+\sqrt{-D}}{2}\right),	 bt+d\left(y+g\dfrac{1+\sqrt{-D}}{2}\right)\right\}$ is a good basis of $I$.		
	\end{enumerate} 
	
	\begin{enumerate}
	\item[\textbf{Step 1}:] Add $(1,0,b,1)$
	into $L$ where $b$ is an integer belonging to $$\left[-\dfrac{1}{2}-\dfrac{2y+g}{2t},\dfrac{1}{2}-\dfrac{2y+g}{2t}\right].$$
	
	\item[\textbf{Step 2}:]
	\begin{enumerate}[2. 1.]
		\item Compute $a=\dfrac{2y+g}{\gcd(2t,2y+g)}$ and $c=-\dfrac{2t}{\gcd(2t,2y+g)}$, then replace $\{a,c\}$ with $\{-c,-a\}$. Using \eqref{(40b)} to compute $\beta_1,\beta_2=\beta_1+a$. Add $(a,c,-b,-d)$ satisfies \eqref{condition3} into $L$ where $d\in[\beta_1,\beta_2]$ such that $1+dc$ is a multiple of $a$ and $b=\dfrac{1+bc}{a}$.
		
		\item If $t\ge 2y+g$, compute $\alpha=\dfrac{t}{2y+g}$ and $\beta_{11}$ and $\beta_{12}$ by using \eqref{(38)} and \eqref{(39)}. Let $\beta_{21}=\beta_{11}+1$ and $\beta_{22}=\beta_{12}+1$. Add all tuples $(0,1,1,d)$ satisfies \eqref{condition3} into $L$ where $d\in[\beta_{11},\beta_{21}]\cup[\beta_{12},\beta_{22}]$. For each integer a in $\left[1,\dfrac{2y+g+1}{2}\right]$, compute all nonzero integers $c$ satisfying \eqref{(45)}, $\gcd(a,c)=1$ and $2at+2cy+cg\ne 0$. Compute $\beta=\dfrac{t}{c(2at+2cy+cg)}$ and $\beta_{11},\beta_{12}$ using \eqref{(42)}. Let $\beta_{21}=\beta_{11}+a,\beta_{22}=\beta_{12}+a$. Add $(a,c,b,d)$ satisfies \eqref{condition3} into $L$ where $b\in[\beta_{11},\beta_{21}]\cup[\beta_{12},\beta_{22}]$ satisfying that $1+bc$ is a multiple of $a$ and $d=\dfrac{1+bc}{a}$.
	
		\item If $t<2y+g$,
	 for each integer a in $\left[1,\dfrac{2y+g+1}{2}\right]$, compute all nonzero integers $c$ satisfying \eqref{(45)}, $\gcd(a,c)=1$ and $2at+2cy+cg\ne 0$. Compute $\beta=\dfrac{t}{c(2at+2cy+cg)}$ and $\beta_{11},\beta_{12},\beta_{21}=\beta_{11}+a,\beta_{22}=\beta_{12}+a$ using \eqref{(42)}. Add $(a,c,b,d)$ satisfies \eqref{condition3} into $L$ where $b\in[\beta_{11},\beta_{21}]\cup[\beta_{12},\beta_{22}]$ satisfying that $1+bc$ is a multiple of $a$ and $d=\dfrac{1+bc}{a}$.
	\end{enumerate}
	\end{enumerate}
\end{algorithm}

%% file: analysis.tex
In this section, we will prove the correctness and the complexity of Algorithms \ref{algo1} and \ref{algo2}.
\subsection{Correctness} 
Let $I$ be an ideal with the canonical basis $\left\{t,y+g\delta\right\}$. We prove that tuples $(a,c,b,d)$ outputted from Algorithm \ref{algo1} form good bases $\left\{z,z'\right\}$ where $z=at+c\left(y+g\delta\right)$ and $z'=bt+d\left(y+g\delta\right)$.
\begin{definition}
	Suppose that $I$ is an ideal of $\mathcal{O}_K$ and $\left\{t, y+g\delta\right\}$ its canonical basis. A \textbf{good tuple} is a tuple $(a,c,b,d)$ of integer numbers satisfying that $\left\{z,z'\right\}$ is a good basis of $I$ where $z=at+c\left(y+g\delta\right)$, $z'=bt+d\left(y+g\delta\right)$. A pair $(a,c)$ is called \textbf{extendable} for $I$ if $z=at+c\left(y+g\delta\right)$ can be extended to a good basis of $I$.
\end{definition}


We prove the correctness of Algorithm \ref{algo1} as below.

\begin{proposition}
	\label{bode8}All tuples $(a,c,b,d)$ in the output of Algorithm \ref{algo1} are  good tuples of $I$.  Inversely, the output of Algorithm \ref{algo1} completely exports all good tuples of $I$, up to similarity.
\end{proposition}

\begin{proof}
We notice that $(a,c,b,d)$ is a good tuple if and only if $(a,c)$ is extendable. By the argument of given in the proof of Lemma \ref{bode3}, the pair $(a,c)$ is extendable if and only if $\gcd(a,c)=1$ and $\left|c(at+cy)\right|\le \dfrac{t}{2}$. Hence, it is trivial to prove this proposition using the computation shown before Algorithm \ref{algo1}.
\end{proof}
Similarly, we can show the correctness of Algorithm \ref{algo2} as below. 

\begin{proposition}
	\label{bode9}All tuples $(a,c,b,d)$ in output of Algorithm \ref{algo2} are good tuples of $I$.  Inversely, the output of Algorithm \ref{algo2} completely exports all good tuples of $I$, up to similarity.
\end{proposition}


\subsection{Complexity}
In this section, we provide an upper bound for the number of loops and operations in Algorithms \ref{algo1} and \ref{algo2}. 

\begin{lemma}\label{bode10}
	The total number of loops of Algorithm \ref{algo1} is at most $y+2$. In addition, an upper bound for the number of operations (expect addition) of each loop in Algorithm \ref{algo1} is $57$.
\end{lemma}
\begin{proof}
It is easy to see the first statement. Since the maximum number of operations in Algorithm \ref{algo1} occurs when $a\in \left[1,\dfrac{y+1}{2}\right]$ and it took us 57 operations (expect addition), the second statement is obtained.
\end{proof}


\begin{lemma}
	\label{bode12} The largest number computed in Algorithm \ref{algo1} is $\dfrac{5t(y+1)+8}{8}$.
\end{lemma}

\begin{proof}
	First, we have $|c|\le \dfrac{t}{2}$ and  
	$a\le \dfrac{y+1}{2}$.
	Using \eqref{(32)} and \eqref{(33)}, one obtains that $ -\dfrac{1}{2}-2\alpha	\le\beta_{11}, \beta_{12} \le -\dfrac{1}{2}$.
 
	In case of applying \eqref{(36)}, since $c<0$ and $a>0$, one has
\begin{align*}
\dfrac{1}{2}\le \left|\dfrac{ac-2}{2c}+a\beta\pm a\sqrt{\beta^2-\dfrac{3}{4}}\right|\le\dfrac{|ac-2|}{2}+2a|\beta|\le \dfrac{t(y+1)+8}{8}+\dfrac{t(y+1)}{2}=\dfrac{5t(y+1)+8}{8}.
\end{align*}
	In the other words, we have the following inequality $\left|\beta_{ij}\right| \le\dfrac{5t(y+1)+8}{8}$. 
	
If we employ \eqref{(20a)},\eqref{(20b)} to compute $\beta_{1}$, then $
\left|\beta_1\right|\le y+1.$\\
Thus $\dfrac{5t(y+1)+8}{8}$ is the maximum number computed in Algorithm \ref{algo1}.
\end{proof}

Similarly, we have the following results for Algorithm \ref{algo2}.


\begin{lemma}
	In Algorithm \ref{algo2}, the total number of loops of  is at most $2y+g+2$ and an upper bound for the number of operations of each loop  is $65$. In addition, the largest number computed in this algorithm  is 
	$\dfrac{5t(2y+g+1)+4}{4}$.
\end{lemma}